\documentclass[11pt]{amsart}
\usepackage{a4wide}

\usepackage[a4paper, margin=2cm]{geometry}
\usepackage{amsmath,amssymb,amsthm, bm}
\usepackage{color}
\usepackage{parskip}
\usepackage{hyperref}
\usepackage[hyphenbreaks]{breakurl}
\usepackage{xurl}
\hypersetup{breaklinks=true}
\usepackage{parskip}
\usepackage{setspace}
\usepackage{graphicx}
\numberwithin{equation}{section}

\usepackage{setspace}
\usepackage{graphicx}
\usepackage{mathtools}

\theoremstyle{plain}
\newtheorem{theorem}{Theorem}[section]

\newtheorem*{corollary*}{Corollary}
\newtheorem{proposition}[theorem]{Proposition}
\newtheorem*{conjecture*}{Conjecture}

\newtheorem{property}{Property}
\newtheorem{corollary}{Corollary}

\theoremstyle{definition}
\newtheorem{definition}[theorem]{Definition}

\theoremstyle{remark}
\newtheorem*{remark*}{Remark}

\newtheorem{case[theorem]}{Case}

\newtheorem*{question*}{Question}

\newtheoremstyle{named}{}{}{\itshape}{}{\bfseries}{.}{.5em}{\thmnote{#3 }#1}
\theoremstyle{named}
\newtheorem*{namedlemma}{Lemma}
\newtheorem*{namedtheorem}{Theorem}

\newtheoremstyle{author}{}{}{\itshape}{}{\bfseries}{.}{.5em}{\thmnote{#1 }(#3)}
\theoremstyle{author}

\newtheorem*{authortheorem}{Theorem}

\usepackage{fixme}
\fxsetup{
    status=draft,
    author=,
    layout=margin,
    theme=color
}

\newcommand{\F}{\mathbb{F}}

\newcommand{\Z}{\mathbb{Z}}
\newcommand{\R}{\mathbb{R}}
\newcommand{\C}{\mathbb{C}}

\newcommand{\cS}{\mathcal{S}}

\newcommand{\cB}{\mathcal{B}}
\newcommand{\fpstar}{\mathbb{F}_p^*}
\newcommand{\eps}{\epsilon}
\def\supp{{\rm supp\,}}
\def\FF{\widehat}

\begin{document}

\title{A new bound for $A(A + A)$ for large sets}
\author{Aliaksei Semchankau}
\date{}
\maketitle

\begin{abstract}
\setlength\parindent{0pt}
For $p$ being a large prime number, and $A \subset \F_p$ we prove the following:
\begin{enumerate} 
\item If $A(A+A)$ does not cover all nonzero residues in $\F_p$, then $|A| < p/8 + o(p)$.
 \item If $A$ is both sum-free and satisfies $A = A^*$, then $|A| < p/9 + o(p)$. 
 \item If $|A| \gg \frac{\log\log{p}}{\sqrt{\log{p}}}p$, then
$|A + A^*| \geqslant (1 - o(1))\min(2\sqrt{|A|p}, p)$.
\end{enumerate}
Here the constants $1/8$, $1/9$, and $2$ are the best possible.
The proof involves \emph{wrappers}, subsets of a finite abelian group $G$, with which we `wrap' popular values in convolutions $A * B$ for dense sets $A, B \subseteq G$.
These objects carry some special structural features, making them capable of addressing both additive-combinatorial and enumerative problems.

\end{abstract}

\section{Introduction}

Szemeredi's Regularity Lemma \cite{szemeredi-regularity} is a key result in structural Graph Theory, having numerous applications in and outside the scope of Graph Theory. We quote it as in \cite{green-regularity}. \\
Let $G = (V, E)$ be a graph and let $A, B$ be disjoint subsets of $V$. Define the density
$d(A, B)$ to be the proportion of elements $(x, y) \in A \times B$ such that $xy \in E$. If $\eps > 0$, we say that a pair $(A, B)$ is $\eps$-uniform if
$|d(A', B') - d(A, B)| \leqslant \eps$
whenever $A' \subseteq A$ and $B' \subseteq B$ satisfy $|A'| > \eps|A|$ and $|B'| > \eps|B|$.

\begin{namedlemma}[Szemeredi's Regularity]
Let $\eps > 0$. There exists $M = M(\eps)$ such that the vertex set $V$ of any graph $G$ can be partitioned into $1/\eps < m < M$ sets $V_1, \ldots, V_m$ with sizes differing by at most $1$, such that at least $(1-\eps)m^2$ of the pairs $(V_i, V_j )$ are $\eps$-uniform.
\end{namedlemma}

This has the following structural corollary, history of which is discussed in \cite{green-regularity}:
\begin{namedlemma}[Triangle Removal]
Let $G$ be a graph on $n$ vertices, and suppose that $G$ contains $o(n^3)$ triangles. Then we may remove $o(n^2)$ edges from $G$ so as to leave a graph which is triangle-free.
\end{namedlemma}

In particular, this implies the Roth-Varnavides Theorem: if a subset $A \subseteq \{1, 2, \ldots, N\}$ has $o(N^2)$ arithmetic progressions of length $3$, then $|A| = o(N)$.\\
In 2005, Green demonstrated (see Theorem 1.5 in \cite{green-regularity} and also the further development of these ideas in \cite{green-tao-regularity}), that a similar structural result takes place in finite abelian groups, giving the following. Let $G$ be an abelian group of size $N$:
\begin{authortheorem}[Green, 2005]
    Let $k \geqslant 3$ be a fixed integer, and suppose that $A_1, A_2, \ldots, A_k$ are subsets of $G$ such that there are $o(N^{k-1})$ solutions to the equation $a_1 + \ldots + a_k = 0$. Then we may remove $o(N)$ elements from each $A_i$, so as to leave sets $A_i'$ such that there are no solutions to $a_1' + \ldots + a_k' = 0$ with $a_i' \in A_i'$ for all $i$.
\end{authortheorem}

This allowed answering in positive the following question of Bergelson, Host, and Kra:\\
%\begin{question*}
{\it Question:}
Suppose that $\alpha, \eps > 0$. Is it true, that there is $N_0(\alpha, \eps)$ such that if $N > N_0(\alpha, \eps)$ and $A \subset \{1, 2, \ldots, N\}$ has size $\alpha N$, then there is some $d \neq 0$ such that $A$ has at least $(\alpha^3 - \eps)N$ three-term arithmetic progressions of length $d$?
%\end{question*}

In this paper we prove the following result, which resembles Triangle Removal Lemma and Green's Theorem:

\begin{namedlemma}[Main] 
\label{lemma:main}
Let $A_1, A_2, \ldots, A_k \subset \F_p$  be such that $|A_i| \gg p$ for all $i$. 
Assume that $(A_1 * A_2 * \ldots * A_k)(a) = o(p^{k-1})$ for some $a \in \F_p$.
Then there exist $Y_1, \ldots, Y_k, W_1, \ldots, W_k \subset \F_p$, such that:
\begin{enumerate}
    \item $(W_1 * W_2 * \ldots * W_k)(b) = o(p^{k-1})$ for some $b \in \F_p$,

    \item $A_i \setminus Y_i \subseteq W_i$ and $|Y_i| = o(p)$ for all $i$, 
    
    \item $|W_i|_{\omega} = p^{o(1)}$ for all $i $, where $|\cdot|_{\omega}$ is a Wiener norm. 
\end{enumerate}
\end{namedlemma}
%\begin{remark*}
{\it Remark:} Point (1) implies $|W_1| + |W_2| + \ldots + |W_k| < p + o(p)$, see Proposition \ref{prop:supersaturation}.\\
%\end{remark*}
%\begin{remark*}
{\it Remark:} Point (3) implies $|(cW_i+e) \cap (dW_j+f)^*| = (1 + o(1))|W_i||W_j|/p$ for all $i, j$ and $c, d, e, f \in \F_p, cd \neq 0$, see Property \ref{property:wcapw}.
%\end{remark*}

This result allows us to give the precise answers to the following questions, which we briefly mention here, see the history and the discussion in Section \ref{sec:results}:\\ 
%\begin{question*}
{\it Question:} Let $A \subset \F_p$ and assume $A(A+A)$ does not cover all nonzero values, how large can $A$ be?\\ 
%{\it Question:}   What is the largest possible size of $A \subseteq \F_p$, such that $A(A+A)$ does not cover all nonzero values?\\
%\end{question*}
%\begin{question*}
 {\it Question:}   What is the largest possible size of $A \subset \F_p$, such that it both sum-free and $A=A^*$?\\
%\end{question*}
%\begin{question*}
 {\it Question:}   Let $A \subset \F_p$ be of size $\alpha p$, what is the smallest possible size of the set $A+A^*$?
%\end{question*}

As we have seen, our new structural Main Lemma allows us to obtain optimal (up to an $o(p)$-term) results, and we deduce it with the help of a new conception of \emph{wrappers} (see Section \ref{sec:wrappers}), subsets of an abelian group $G$, approximating level-sets of sumsets $A+B$, and, more generally, level-sets of any function $f$, having a small Wiener norm.
A wrapper comes with two parameters: $\eps \in (0, 1)$, which we call its \emph{granularity}, and positive integer $d$, which we call its \emph{dimension}. We write $\eps$-$d$-wrapper for short. Granularity and dimension of wrappers are exploited in Properties 2 and 3 in Section \ref{sec:discussion} below.\\
Szemer{\'e}di Regularity Lemma and Green's Theorem proofs rely on an energy increment argument, which gives tower-like dependencies, which makes results only applicable for graphs/sets of positive density (or, more precisely, for density of order at least $1/\log^{*}|G|$, where $G$ is a graph/group, and $\log^{*}$ is a number of times one has to logarithmize the argument to obtain a value below $2$).
However, at the cost of succinctness in the statement, we can prove the following quantitative result (at the further cost of symmetry, one can also obtain Technical Lemma below), which is applicable for sets of density $1/(\log|G|)^c$, $c > 0$. We defer both proofs of Main and Supplementary Lemma to Section \ref{sec:main_supp_proofs}.

\begin{namedlemma}[Supplementary]
\label{lem:supp}
Let $G$ be a finite abelian group, and let $\alpha, \beta \in (0, 1)$. Let $A, B \subset G$, and assume $|A| \geqslant \alpha|G|, |B| \geqslant \beta|G|$.
Let $\eta_1, \eta_2 \in \R, \delta, \xi \in \R^{+}$.
Let 
$$
X := \{x : \eta_1 |G| \leqslant (A * B)(x) \leqslant \eta_2 |G|\}, \  
X^{+} := \{x : (\eta_1 - \delta)|G| < (A * B)(x) < (\eta_2 + \delta)|G|\}.
$$
Then there exist wrapper $W$ of granularity $\eps = \eps(\alpha, \beta) \gg \min\big(\frac{\delta}{\sqrt{\alpha\beta}}, 1\big)$ and dimension $d = d(\alpha, \beta, \xi) \ll \max\big(\frac{\alpha \beta}{\delta^2}, 1\big) \log{\big(\frac{1}{\xi}\big)}$, and set $Y$, $|Y| \leqslant \xi p$, satisfying inclusions $X \setminus Y \subseteq W$ and $W \setminus Y \subseteq X^{+}$, which we call wrapping.

\end{namedlemma}

The rest of the paper is structured as follows. \\
In Section \ref{sec:notations} we introduce the basic notation and recall useful facts on Discrete Fourier Transform.\\ 
In Section \ref{sec:discussion} we provide some further discussion on wrappers.\\
In Section \ref{sec:results} we state Theorems \ref{thm:a(a+a)}, \ref{thm:a=a_inverse}, \ref{thm:a+a_inverse}\\
In Section \ref{sec:aux} we provide some auxiliary results on sumsets.\\
In Section \ref{sec:applications} we apply Main and Supplementary Lemmas to obtain the Theorems from \ref{sec:results}.\\
In Sections \ref{sec:wrappers} we introduce wrappers, which serve to prove Main and Supplementary Lemmas in Section \ref{sec:main_supp_proofs}.

\section{Acknowledgements}
This work was supported by the Russian Science Foundation under grant no.19-11-00001, \url{https://rscf.ru/en/project/19-11-00001/}. The author is very grateful to Ilya Shkredov for suggesting the problem on size of $A(A+A)$ as a project, and his constant attention to this work.

\section{Notation}
\label{sec:notations}

Throughout the paper, the standard notation $\ll, \gg$, and respectively $O$ and $\Omega$ is applied to positive quantities in the usual way. That is, $X \ll Y, Y \gg X, X = O(Y)$ and $Y = \Omega(X)$ all mean that $Y \geqslant cX$, for some absolute constant $c > 0$.  \\
$G$ always means a finite abelian group, and $p$ always means a large prime number.\\
For $A \subseteq G$ denote by $A$ its characteristic function $A : G \rightarrow \C$. For $A \subseteq \F_p$ we denote by $A^*$ the set $\{1/a \bmod p\ |\ a \neq 0 \}$.\\
Given $A, B \subseteq G$, define their \emph{sumset} by $A+B := \{a + b\ |\ a\in A, b \in B\}$. Given also $\eps > 0$, define their \emph{partial sumset} by
$
A+_{\eps}B := \{ x \in A + B : (A * B)(x) \geqslant \eps |G|\}.
$
Denote by $\overline{A}$ the set $G\setminus A$.
We write $A \neq B$ if $A \cap B = \varnothing$. 
Clearly, $A \neq \overline{A}$ for any $A \not\in \{\varnothing, G\}$.\\
For given $G$ we write $\hat{G} = \{\gamma : G \rightarrow \cS\ |\ \gamma \text { a homomorphism} \}$, where $\cS = \{ z \in \C : |z| = 1\}$. This $\hat{G}$ is called a \emph{dual group} of $G$. Note that for $G = \Z_N$ the dual group $\hat{\Z}_N$ consists of the functions $\gamma_r(x) := e^{2\pi i r x / N}$ for $ r \in \Z_N$.\\
Given functions $f, g : G \rightarrow \C$, define the \emph{convolution} $f * g : G \rightarrow \C$ by
$(f * g)(x) := \sum_{a, b : a+b = x}f(a)g(b). $\\
We define the Fourier transform $\hat{f} : \hat{G} \rightarrow \C$ by
$
\hat{f}(\gamma) := \sum_{x \in G}f(x)\overline{\gamma(x)}.
$
We assume the following properties to be well-known:
$$
f(x) = 
\frac{1}{|G|}
\sum_{\gamma \in \hat{G}}
\hat{f}(\gamma)\gamma(x), \ \ 
\sum_{x \in G
}f(x)\overline{g(x)} = 
\frac{1}{|G|}
\sum_{\gamma \in \hat{G}}
\hat{f}(\gamma)
\overline{\hat{g}(\gamma)},
\ \ 
\FF{f * g}(\gamma) = 
\hat{f}(\gamma) 
\hat{g}(\gamma),\ \ 
$$
which are the Inverse Fourier Transform, Parseval Identity, and Convolution Identity, respectively.\\
Denote by $||f||_{\omega}$ the \emph{Wiener norm}:
$||f||_{\omega} := \frac{1}{|G|}\sum_{\gamma \in \hat{G}}|\hat{f}(\gamma)|.$

\section{Discussion on Wrappers} \label{sec:discussion}
Here are the useful properties of the wrappers to be used in applications. Their proofs are deferred to Section \ref{sec:wrappers}.
\begin{property}
Let $W \subseteq G$ be an $\eps$-$d$-wrapper. Then its complement $\overline{W}$ is an $\eps$-$d$-wrapper, too.
\end{property}

\begin{property}
\label{property:wiener_bound}
Let $W \subseteq \F_p$ be an $\eps$-$d$-wrapper. Then
$
|W|_{\omega} \leqslant (C\log{p}/\eps)^d, 
$
where $C > 0$ is absolute.
\end{property}

\begin{property}
\label{property:wcapw}
Let $W_1, W_2 \subseteq \F_p$ be of Wiener norms $\omega_1, \omega_2$, respectively. Then
$$
|W_1 \cap W_2^*| = \frac{|W_1||W_2|}{p} + O(\omega_1 \omega_2\sqrt{p}).
$$
\end{property}

Let us mention the connection between wrappers and graph containers. Graph Container theorems were independently introduced by Saxton, Thomason (see \cite{saxton-thomason}) and Balogh, Morris and Samotij (see \cite{balogh-morris-samotij}), and had a remarkable impact on extremal combinatorics in recent years. \\
Let $\mathcal{H}$ be a $3$-uniform hypergraph. 
Denote by $V(\mathcal{H})$ the set of vertices, and by $E(\mathcal{H})$ the set of edges.
For any $V' \subset V$ denote by $\mathcal{H}[V']$ the subgraph induced by $V'$. 
Define by $d(\mathcal{H})$ the average vertex degree in $\mathcal{H}$, and by $\Delta_2(\mathcal{H})$ the maximal number of edges two particular vertices might belong to. Then, Corollary 3.6 \cite{saxton-thomason} gives roughly the following: 
\begin{namedtheorem}[$3$-Uniform Graph Container]
    Let $\mathcal{H} = (V, E)$ be a $3$-uniform hypergraph on $n$ vertices. Let $\eps, \tau$ be sufficiently small positive constants. 
    Suppose that $\Delta_2(\mathcal{H}) \ll \eps \tau d(\mathcal{H})$.  Then there exists a set $\mathcal{C}$ of subsets in $V$ such that
    \begin{enumerate}
        \item If $A \subset V$ is an independent set, then $A \subseteq C$ for some $C \in \mathcal{C}$,
        \item $|E(\mathcal{H}[C])| \leqslant \eps |E(\mathcal{H})|$ for any $C \in \mathcal{C}$,
        \item $\log{|\mathcal{C}|} \ll n\tau \log{(\frac{1}{\eps})}\log{(\frac{1}{\tau})}$.
    \end{enumerate}
\end{namedtheorem}
The set $\mathcal{C}$ above is referred to as a set of containers, and a set $C \in \mathcal{C}$ is itself a container.

One can notice from the proof of Supplementary Lemma, that there are just $O(2^{(1/\eps)^d}p^d)$ wrappers with given parameters $\eps, d$, which makes us able to formulate the result of the following form:
\begin{namedlemma}[Container Form of Supplementary]
    Let $\kappa(p) := p/(\log{p})^{c}$, where $c > 0$ is some absolute constant, which we do not state here. 
    Given $\chi > 0$, there exists $\mathcal{F}_{\chi} \subseteq 2^{\F_p}$ with $\log{|\mathcal{F}_{\chi}}| \ll \kappa(p)$ such that the following holds.\\
    Let $A, B, C \subseteq \F_p$, $|A|, |B|, |C| \geqslant \chi p$ be such that $A \neq B + C$. Then there exists $X_A, X_B, X_C \subseteq \mathcal{F}_{\chi}$ such that $A \subseteq X_A, B \subseteq X_B, C \subseteq X_C$ and $|X_A| + |X_B| + |X_C| \leqslant p + O(\kappa(p))$. 
\end{namedlemma}
Using a similar observation, we find a correct estimate for a number of triples $A, B, C \in 2^{G}$ such that $A \neq B + C$ in \cite{SSS}. We are planning to exploit this idea further in a follow-up paper.

\section{Results}
\label{sec:results}

\subsection{Covering $\fpstar$ by $A(A+A)$}

The first problem we are interested in is to determine a size of the set $A(A+A) = \{x(y+z), x,y,z\in A\}$ for $A \in \fpstar$. This sort of problems belongs to a range of expanding polynomials and sum-product problems. In the literature, they are usually discussed in the sparse-set regime; for instance, Roche-Newton, Rudnev, and Shkredov \cite{sum-prod}, Yazici, Murphy, Rudnev, and Shkredov \cite{growth} proved that in the regime $|A| \ll p^{2/3}$ one has $\min{(|A+AA|, |A(A+A)|)} \gg |A|^{3/2}$ (see also \cite{e-p}). This implies, in particular, that as soon as $|A| \gg p^{2/3}$ both sets $A(A+A)$ and $A+AA$ occupy a positive proportion of $\F_p$.\\
Now we focus on a case when $A$ already occupies a positive proportion of $\F_p$. Let $\alpha = |A|/p$, and we suppose $\alpha$ is bounded below by a positive constant when $p$ tends to infinity. Bienvenu, Hennecart, and Shkredov \cite{bhs19} proved that in such a regime $A(A+A)$ contains all but a finite number of elements. That is, when $A \subseteq \fpstar, |A| = \alpha p, 0 < \alpha < 1$ one has $|A(A+A)|~>~p~-~1~-~\alpha^{-3}(1-\alpha)^2~+~o(1)$. Besides that, they proved that this finite number of elements may be strictly larger than $1$, unless $\alpha$ is large enough. In particular, for any $k > 0$ there exists $A \subset \fpstar, |A| = (1/8k + o(1))p$, and $|A(A+A)| \leqslant p - 1 - k$.
In contrary, they demonstrated the following:
\begin{authortheorem}[Bienvenu-Hennecart-Shkredov, 2019]
Let $A \subseteq \F_p$ be such that $A(A+A)$ does not cover all nonzero residues. Then $|A| < 0.3051p + o(p)$.
\end{authortheorem}
We improve it as following:
\begin{theorem}
\label{thm:a(a+a)}
Let $A \subseteq \F_p$ be such that $A(A+A)$ does not cover all the nonzero residues. Then $|A| \leqslant p/8 + o(p)$.
\end{theorem}
{\it Remark.}
Consider sets 
$P = \{x: 0 < x < p/4\}$ and 
$Q = \{x: p/2 < x < p\}$. 
Set $A := P \cap Q^*$.
Clearly, $|A| = p/8 + o(p)$. 
Now observe $(A + A) \cap A^* = \varnothing$, and therefore $1 \not\in A(A+A)$, which implies that threshold $\alpha = 1/8$ is indeed optimal.

\subsection{On size of a sum-free $A$ with $A = A^*$}
$A \subseteq \F_p$ is called \emph{sum-free} if $A$ does not intersect its sumset $A + A$. From Cauchy-Davenport inequality  size of such $A$ cannot exceed $(p + 1)/3$, with equality when $A$ is (some dilation of) a segment $\{x: p/3 < x < 2p/3 \}$. As a generalization of this fact, it was shown in \cite{lev07} that the structure of a sum-free set $A$ with size close to $p/3$ looks like a segment of length $p/3$: namely, if $|A| > 0.318p$, then $A$ belongs to (some dilation of) the segment $\{|A|, \ldots, p - |A|\}$.\\
This gives an idea that the structure of a sum-free $A$ should be `arithmetical' in nature. It was conjectured by Benjamin in \cite{ben20} that events `$A$ is sum-free' and `$A = A^*$' should be independent in a sense that size of a sum-free $A$ with $A = A^*$ does not exceed $p/9 + o(p)$.\\
Bound in \cite{bhs19} shows already that size of such $A$ does not exceed $0.3051p$. Furthermore, the following was proved in \cite[Theorem 1.1]{ben20}:
\begin{authortheorem}[Benjamin, 2020]
There is an absolute constant $c > 0$ so that if $A \subseteq \F_p$ is sum-free and closed under inverses then $|A| < (0.25 - c)p + o(p)$.
\end{authortheorem}
One can deduce from the proof, that the value $c = 1/40\,000\,000$ is admissible.
We prove the following:
\begin{theorem}
\label{thm:a=a_inverse} 
If $A \subseteq \F_p$ is sum-free and closed under inverses, then $|A| < p/9 + o(p)$.
\end{theorem}
{\it Remark.}
Let $P := \{ x : p/3 < x < 2p/3\}$ Set $A := P \cap P^*$. Clearly, $|A| = p/9 + o(p)$, $A$ is sum-free, and $A = A^*$. Therefore, inequality $|A| \leqslant p/9 + o(p)$ is indeed optimal.

\subsection{On size of the set $A + A^*$}

The third problem we are interested in is to study the cardinality of the set $A + A^{*}$, $A \subseteq \fpstar$, where $A^*$ means the set of inverted elements of $A$, and $p$ is a large enough prime number.\\
Let $P := \{ x : 0 < x < \alpha p\}, A := P \cap P^{*}$. Clearly,
$|A| = (1 + o(1))\alpha^2 p$
and 
$|A + A^*| = (1 + o(1))\min(2\alpha p, p)$, which suggests the following: 

\begin{conjecture*}
There exists $\eps > 0$, such that for any $A \subseteq \fpstar, |A| \gg p^{1 - \eps}$
$$
|A+A^*| > (1 - o(1))\min{(2\sqrt{|A|p}, p)}.
$$
\end{conjecture*}

Best estimates for $A+A^*$ in different regimes of size of $|A| \sim p^{\lambda}$:

\begin{center}
\begin{tabular}{ |c|c|c|c| } 
\hline
$ \lambda \in [0, \frac{15}{29}]$ & $|A+A^*| \gg |A|^{31/30}$ & Yazici, Murphy, Rudnev, Shkredov & (Prop. 14, \cite{growth})\\ 
\hline
$\lambda \in [\frac{15}{29}, \frac{2}{3}]$ & $|A+A^*| > (1/2-o(1))|A|^2/\sqrt{p}$ & Balog, Broughan, Shparlinski & (Thm 7, \cite{a+a*})\\
 
\hline
$ \lambda \in [\frac{2}{3}, 1] $ & $|A+A^*| > (1-o(1))\sqrt{p|A|}$ & Balog, Broughan, Shparlinski & (Thm 7, \cite{a+a*})\\
 
\hline
\end{tabular}
\end{center}

We prove Conjecture above in case when $A$ is dense.
\begin{theorem}
\label{thm:a+a_inverse}
Let $A \subseteq \fpstar, |A| = \alpha p, \alpha \gg \kappa(p)$, where $\kappa(p) := \frac{\log\log{p}}{\sqrt{\log{p}}}$. Let $\delta$ be such that
$
\alpha^{3/2} \gg \delta \gg \alpha^{1/2}\kappa(p).
$
Then
$
|A +_\delta A^*| 
\geqslant 
\min(2\sqrt{|A|p},p) - O(\sqrt[3]{\delta} p)
$
\end{theorem}

\section{Auxiliary Lemmas}
\label{sec:aux}

Recall the Cauchy-Davenport Inequality. See \cite{nathanson}, Theorem 2.2 for a proof.
\begin{proposition}
Let $A, B$ be subsets of $\F_p$. Then $|A+B| \geqslant \min{(p, |A|+|B|-1)}$.
\end{proposition}
It has a following generalization, which might be deduced from the result of Pollard (see Equation 3 in \cite{pollard}) combined with the argument of Tao (see Corollary 1.2 in \cite{tao}).
\begin{proposition}
\label{prop:pollard}
Let $A, B$ be subsets of $\F_p$. Let $0 < \eps < \min{(\alpha^2, \beta^2)}$. Then 
$
|A+_{\eps}B| \geqslant \min{(p, |A|+|B|)}-(2+o(1))\sqrt{\eps}p.
$
\end{proposition}
 If $A, B, C \subset G$ are arbitrary nonempty sets and $A \neq B + C$, then $B \neq A - C$. This has a following generalization:
\begin{proposition}
\label{prop:struct}
Let $A, B, C \subset G$ be nonempty sets, $|A| = \alpha |G|, |C| = \gamma |G|$, and $\delta \in (0, 1)$. Suppose that $A \neq B +_{\delta} C$, then for any $T > 1$ there exists $E$, $|E| \leqslant |C| / T$, such that 
$
C\setminus E \neq A -_{\alpha \delta T / \gamma} B.
$
\end{proposition}
\begin{proof}
Let $E := C \cap (A -_{\eta} B)$, $\eta := \alpha \delta T / \gamma$.
Clearly,
$
{|E|\eta |G| \leqslant \#\{(a, b, c) \in A \times B \times C : a = b + c\} \leqslant |A|\delta |G|}.
$
This implies $|E| \leqslant |C|/T$, which completes the proof.
\end{proof}

\begin{proposition}\label{prop:supersaturation}
Let $\eta > 0$ and $X_1, \ldots, X_k \subseteq \F_p$ be such that $|X_i| \gg \eta p$, and $|X_1| + \ldots + |X_k| \geqslant p + \eta p$.
Then $(X_1 * \ldots * X_k)(0) \gg \eta^{2k-3}p^{k-1}$.
\end{proposition}
\begin{proof} 
    Let $\beta$ be a smallest density among sets $X_1, \ldots, X_k$ and let $\eps := \min(0.1\beta^2, \eta^2/16k^2)$.
    Let us set $V_1 := X_1$, and $V_{i} := V_{i-1} +_{\eps} X_{i}$ for $i = 2, \ldots, k-1$. 
    By an induction argument, this is clear that $|V_i| \geqslant |X_1| + \ldots + |X_i| - 3i\sqrt{\eps}p \geqslant \beta p$. Indeed, this holds for $i=1$, and for larger $i$ we have $|V_i| \geqslant |V_{i-1}| + |X_i| - (2+o(1))\sqrt{\eps}p$ by Proposition \ref{prop:pollard}.
    This is clear that 
    $$
    (X_1 * X_2 * \ldots * X_k)(0) = (V_1 * X_2 * \ldots X_k) (0) \geqslant
    (\eps p)(V_2 * \ldots * X_k)(0) \geqslant 
    \ldots \geqslant 
    $$
    $$
    \geqslant
    (\eps p)^{k-2}(V_{k-1} * X_k)(0) \geqslant 
    (\eps p)^{k-2}(|V_{k-1}| + |X_k| - p) \geqslant 
    (\eps p)^{k-2}(|X_1| + \ldots + |X_k| - p - 3k\sqrt{\eps}p) \geqslant  
    \eps^{k-3/2}p^{k-1},
    $$
    which completes the proof.
\end{proof}

\section{Main and Supplementary Lemma Applications}
\label{sec:applications}

\subsection{Proof of Theorem \ref{thm:a(a+a)}}
\begin{proof}
Without loss of generality we assume that $|A| \gg p$ and $1 \not\in A(A+A)$, hence $A^* \neq A + A$, and therefore the equation $x + y + z = 0$ has no solutions in $(x, y, z) \in -A^* \times A \times A$. Main Lemma application gives sets $W_i, Y_i$, such that $A^*\setminus Y_1 \subseteq W_1$, $A \setminus Y_2 \subseteq W_2, A \setminus Y_3 \subseteq W_3$. Denote densities of $W_i$ by $\omega_i$. Then $\omega_1 + \omega_2 + \omega_3 < 1 + o(1)$. From $A \setminus (Y_1^* \cup Y_2) \subseteq W_1^* \cap W_2$ we obtain
$$
|A| - |Y_1| - |Y_2| \leqslant |W_1^* \cap W_2| = (\omega_1 \omega_2 + o(1))p,
$$
which gives $(1 - o(1))\alpha \leqslant \omega_1 \omega_2 + o(1)$. Similarly, $\alpha \leqslant \omega_1 \omega_3 + o(1)$, from where 
$$
\alpha \leqslant
\omega_1 \frac{\omega_2 + \omega_3}{2} + o(1) \leqslant
\omega_1 \frac{1 - \omega_1}{2} + o(1) \leqslant
\frac{1}{8} + o(1).
$$
\end{proof}
\subsection{Proof of Theorem \ref{thm:a=a_inverse}}
Since $A$ is sum-free, the equation $x + y + z = 0$ has no solutions in $(x, y, z) \in -A \times A \times A$.
Main Lemma applicaton gives $W_i, Y_i$ such that $A \setminus Y_i \subseteq W_i$ for $i = 1, 2, 3$, and densities $\omega_i$ satisfy $\omega_1 + \omega_2 + \omega_3 < 1 + o(1)$.
By pigeonhole principle, some density (say, $\omega_3$) is at most $1/3 + o(1)$. \\
From $A = A^*$ it follows that $A \setminus (Y_3^* \cup Y_3) \subseteq W_3 \cap W_3^*$, which results in $\alpha \leqslant \omega_3^2 + o(1) \leqslant 1/9 + o(1)$.
 
\subsection{Proof of Theorem \ref{thm:a+a_inverse} (qualitative  version)}

Let $B := \overline{A + A^*}$ and $\beta := |B|/p$. 
We need to prove $\beta \leqslant 1- 2\sqrt{\alpha} + o(1)$.
Clearly, equation $x + y + z = 0$ has no solutions in $(x, y, z) \in A \times A^* \times -B$. Main Lemma application gives $W_i, Y_i$ such that $A \setminus Y_1 \subseteq W_1, A^* \setminus Y_2 \subseteq W_2, B \setminus Y_3 \subseteq W_3$, and densities $\omega_i$ satisfy $\omega_1 + \omega_2 + \omega_3 < 1 + o(1)$.\\
From $A \setminus (Y_1 \cup Y_2^*) \subseteq W_1 \cap W_2^*$ we obtain $\alpha(1 - o(1)) \leqslant \omega_1 \omega_2 + o(1)$, which implies $\sqrt{\alpha} + o(1) \leqslant \sqrt{\omega_1 \omega_2}$. 
Therefore, $\beta \leqslant \omega_3 + o(1) \leqslant 1 - \omega_1 - \omega_2 + o(1) \leqslant 1 - 2\sqrt{\omega_1 \omega_2} + o(1) \leqslant 1 - 2\sqrt{\alpha} + o(1)$.

\subsection{Proof of Theorem \ref{thm:a+a_inverse} (quantitative  version)}

\begin{namedlemma}[Technical]
\label{lem:tchn}
Let $X_1, X_2, X_3$ be sets in $\F_p$ of densities $\theta_1, \theta_2, \theta_3$, such that the equation $x + y + z = 0$ has no solutions in $(x, y, z) \in X_1 \times X_2 \times X_3$.
Let $T \gg 1$ and $\delta \in (0, 1)$ be some parameters. 
Assume that for some absolute constant $C > 0$ inequalities $\theta_1, \theta_2, \theta_3, 1/T, \delta > (\log{p})^{-C}$ hold.
Then there exists 
$\eps_1$-$d_1$-wrapper $W_1$,
$\eps_2$-$d_2$-wrapper $W_2$, 
sets $E_1, E_2$, 
such that 
$$
X_i \setminus E_i \subseteq W_i,\ 
|E_i| < |X_i| / T,\  \log{(1/\eps_i)} \ll \log\log{p},
$$
$$
d_1 \ll \max\bigg( \frac{\theta_2 \theta_3}{\delta^2}, 1 \bigg)\log\log{p}, \ \ 
d_2 \ll \max\bigg( \frac{w_1 \theta_3}{\eta^2}, 1 \bigg)\log\log{p},
$$
and
$- (W_2 \setminus E_2) \neq (W_1 \setminus E_1) +_{\eta} X_3$,
where $w_1 := |W_1|/p$, and $\eta = \Theta( w_1 \delta T / \theta_2)$. All the $\ll$ signs assume dependence on $C$.
\end{namedlemma}

\begin{proof}
Clearly, $-X_1 \neq X_2+_{\delta}X_3$.
Let $\xi := (\log{p})^{-10C}$. Clearly, $\xi$ is much lesser than $1/T$.
Supplementary Lemma application (with $\delta$ in place of both $\delta, \eta$) gives  $\eps_1$-$d_1$-wrapper $W_1$ and set $Y_1$ such that 
$$
X_1 \setminus Y_1 \subseteq W_1, \ \ 
- (W_1 \setminus Y_1) \neq X_2 +_{2\delta} X_3, \ \ 
|Y_1| \leqslant \xi p,
$$
$$
\eps_1 \gg \min{\bigg(\frac{\delta}{\sqrt{\theta_2 \theta_3}}, 1\bigg)}, \ \ 
d_1
\ll 
\max{\bigg(\frac{\theta_2 \theta_3}{\delta^2}, 1\bigg)}\log\log{p}.
$$

Let $|W_1| = w_1 p$. 
By the Proposition \ref{prop:struct} there exists $E, |E| \leqslant |X_2|/2T$, such that 
$-(X_2\setminus E) \neq (W_1 \setminus E_1) +_{\eta}X_3 $, where $\eta := 4w_1 \delta T/\theta_2$. Another application of Supplementary Lemma gives $\eps_2$-$d_2$-wrapper $W_2$ and set $Y_2$ such that 
$$
X_2 \setminus (E \cup Y_2) \subseteq W_2, \ \ 
-(W_2 \setminus E_2) \neq (W_1 \setminus Y_1) +_{2\eta}X_3, \ \ , 
\ \ 
|E \cup Y_2| \leqslant |X_2|/2T + \xi p < |X_2|/T,
$$
$$
\eps_2 \gg \min{\bigg(\frac{\eta}{\sqrt{w_1 \theta_3}}, 1\bigg)}, \ \ 
d_2 
\ll 
\max{\bigg(\frac{w_1 \theta_3}{\eta^2}, 1\bigg)}\log\log{p},
$$
which completes the proof.
\end{proof}

\begin{proof}
Let $B := \overline{A + A^*}$. Let $\beta := |B|/p$. 
Suppose that $\beta > 1 - 2\sqrt{\alpha}$, since otherwise there is nothing to prove.
Clearly, equation $x + y  + z = 0$ has no solutions in $x \in A, y \in A^*, z \in -B$. 
Therefore, the Technical Lemma is applicable with $X_1 := A, X_2 := A^*, X_3 := -B$. 
Let $T := c\sqrt[6]{\alpha^3 / \delta^2}$, where $c > 0$ is sufficiently small. Clearly, inequality $T \gg 1$ is satisfied. 
%Let $T$ be some parameter to be chosen later.
We obtain $\eps_1$-$d_1$-wrapper $W_1$, $\eps_2$-$d_2$-wrapper $W_2$, and $E_1, E_2$, such that
$
A\setminus E_1 \subseteq W_1, \ 
A^* \setminus E_2 \subseteq W_2,
$
$
|E_i| < |A|/T,\ 
\log{(1/\eps_i)} \ll \log\log{p}
\text{ for } i \in \{1, 2\},
$
and 
$$
   d_1 \ll \max\bigg( \frac{\alpha}{\delta^2}, 1 \bigg)\log\log{p}, \ 
   d_2 \ll \max\bigg( \frac{w_1}{\eta^2}, 1 \bigg)\log\log{p},
$$
 and $-(W_2 \setminus E_2) \neq (W_1 \setminus E_1) + _{\eta}X_3$, where $\eta = \Theta(w_1 \delta T / \alpha)$, $w_1 := |W_1| / p$.\\
 Let us assume first that $w_1 \gg \sqrt{\alpha}$.\\
To enforce $W_1, W_2$ to have Wiener norms at most $p^{0.001}$, we need $d_1 \ll \log{p}/\log\log{p}$ (which is satisfied by $\delta \gg \kappa(p)\sqrt{\alpha}$), and $d_2 \ll \log{p} / \log\log{p}$ (which is satisfied by the fact that $\omega_1 \gg \sqrt{\alpha}$ and $T^4 \gg \frac{\alpha^3}{\delta^4}\frac{(\log\log{p})^4}{(\log{p})^2} \Leftrightarrow \delta^8 \gg \alpha^3 \kappa(p)^{12}$) 
%so we require $\delta, T$ to be such that the following inequalities hold:
%$$
%\frac{\alpha}{\delta^2} \ll \frac{\log{p}}{(\log\log{p})^2}, \ \ \ \ 
%\frac{w_1}{\eta^2} \ll \frac{\log{p}}{(\log\log{p})^2} 
%\Leftarrow 
%T^4 \gg \frac{\alpha^3}{\delta^4}\frac{(\log\log{p})^4}{(\log{p})^2}.
%$$

Since Wiener norms are small, $A \setminus (E_1 \cup E_2^*) \subseteq W_1 \cap W_2^*$ implies the inequality $\alpha(1 - 2/T) \leqslant w_1 w_2(1 + O(p^{-\eps}))$.\\
To make Proposition \ref{prop:pollard} applicable to $(W_1 \setminus E_1) +_{\eta}B$, we want to enforce  $\eta \ll \min(w_1^2, \beta^2)$. This is satisfied by the fact $T^2 \ll \alpha^3 / \delta^2$. Thus,
$$
|W_2 \setminus E_2| < p - |(W_1\setminus E_1) +_{\eta} B| < p - |W_1\setminus E_1| - |B| + 2\sqrt{\eta}p, 
$$
from where 
$
\beta + w_1 + w_2 < 1 + O(\sqrt{\eta}) + O(\alpha / T).
$
We can assume that both $O(\sqrt{\eta})$ and $O(\alpha / T)$ are sufficiently smaller than $w_1$, and so $w_1 + w_2 \ll 1 - \beta \ll \sqrt{\alpha}$ by the assumption made in the beginning.\\
Therefore from now on we can assume $w_1 \ll \sqrt{\alpha}$, and $\eta \ll w_1 \delta T/\alpha \ll \delta T/\sqrt{\alpha}$.\\
Writing same inequality again, we obtain
$$
\beta < 1 - w_1 - w_2 + O(\sqrt{\eta}) + O(\alpha / T) < 
1- w_1 - w_2 + O(\sqrt{\delta T}/\sqrt[4]{\alpha}) + O(\alpha / T) < 
$$
$$
< 1 - 2\sqrt{\omega_1 \omega_2} + 
O(\sqrt{\delta T}/\sqrt[4]{\alpha}) + O(\alpha / T) < 
1 - 2\sqrt{\alpha} + O(\sqrt{\alpha}/T) + 
O(\sqrt{\delta T}/\sqrt[4]{\alpha}).
$$

The choice $T = c\sqrt[6]{\alpha^3 / \delta^2}$ gives error-term $O(\delta^{1/3})$, as declared.

\medskip

Now let us comment on why we could have assumed $w_1 \gg \sqrt{\alpha}$. One can apply Technical Lemma with $X_1 := A, X_2 := A^*, X_3 := B$ to obtain $\eps_1$-$d_1$-wrapper $W_1$ to wrap $A$, or with $X_1 := A^*, X_2 := A, X_3 := B$ to obtain $\eps_1'$-$d_1'$-wrapper $W_1'$ to wrap $A^*$. Let us denote their densities by $w_1, w_1'$, respectively. Then, since $d_1, d_1' \ll \alpha / \delta^2 \ll \log{p}/(\log\log{p})^2$, Wiener norms of $W_1, W_1^*$ can be considered polynomially small, and therefore $A \setminus (E_1 \cup E_1'^*)\subseteq W_1 \cap W_1'^*$ implies $\alpha \leqslant w_1 w_1'$. We assume without loss of generality that $w_1 \gg \sqrt{\alpha}$, which completes the proof.
\end{proof}

\section{Wrappers}
\label{sec:wrappers}

\subsection{Definition and main result on wrappers}

Let $\eps \in (0, 1/2)$ be such that $K := 1/\eps$ is a positive integer, and let $\gamma_1, \ldots, \gamma_d : G \rightarrow \cS$ be some homomorphisms (not necessarily different). 
Split circle $\cS$ into $K$ equal arcs: $\cS_1, \ldots, \cS_K$. 

\begin{definition}

An \emph{$\eps$-$d$-block} is a set of the form
$$
\bigg\{ 
g \in G : 
\gamma_1(g) \in \cS_{i_1}, 
\ldots, 
\gamma_d(g) \in \cS_{i_d}
\bigg\}
$$
 for some $1 \leqslant i_1, \ldots, i_d \leqslant K$.
\end{definition}

\begin{definition}
Assume a set of $d$ homomorphisms $\gamma_1, \ldots, \gamma_d$ to be fixed. We define $\eps$-$d$-\emph{wrapper} to be a union of some $\eps$-$d$-blocks.
\end{definition}
Clearly, distinct $\eps$-$d$-blocks are disjoint and their union constitutes the whole $G$. This proves, that complement of a wrapper is a wrapper, too.\\
The following statement is the main result on wrappers, with which one can `approximate' level-sets of functions by $\eps$-$d$-blocks with a certain precision.

\begin{namedtheorem}[Wrapper]
\label{thm:decomposition}
Let $f : G \rightarrow \C$ be a function with Wiener norm $\omega$. Let $\xi \in (0, 1/100)$ be some parameter. Let $\delta \leqslant \omega$ be some positive real number. 
Then there exist $\eps \in (0, 1/2)$, positive integer $d$, and  a decomposition 
$$
f = g + h + k,
$$
such that:
\begin{itemize}
    \item $\|h\|_{\infty} \leqslant \delta/10$,
    \item $|\supp{k}| \leqslant \xi |G|$.
    \item There is a disjoint union of $G = \sqcup B_i$ into $\eps$-$d$-blocks, such that $g$ is constant on each block $B_i$. 
    \item $\eps \gg \delta / \omega$, $d \ll \frac{\omega^2}{\delta^2}\log{\frac{1}{\xi}}$.
\end{itemize}
\end{namedtheorem}
\begin{proof}
This is a Corollary 3.6 from \cite{cls}:
\begin{namedlemma}[Croot-Laba-Sisask]
\label{lem:cls}
Let $G$ be a finite abelian group and let $\eps \in (0, 1)$ and $q > 2$ be parameters. For any nonzero function $f : G \rightarrow \C$ of Wiener norm $\omega$ there are characters $\gamma_1, \gamma_2, \ldots, \gamma_d \in \hat{G}$ with $d \leqslant Cq/\eps^2$ and coefficients $c_1, \ldots, c_d \in \C$ with $|c_j| \leqslant 1$ such that 
$$
\|f/\omega - \frac{1}{d}(c_1\gamma_1 + \ldots + c_d\gamma_d)\|_{L^q} \leqslant \eps.
$$
\end{namedlemma}

As usual, \emph{$L_q$-norm} of function $f : G \rightarrow \C$ is defined by $\|f\|_{L_q} := \bigg(\frac{1}{|G|}\sum_{x}|f(x)|^q \bigg)^{1/q}$.\\
Apply Croot-Laba-Sisask Lemma with the parameters $f := f$, $\eps := c\frac{\delta}{\omega}$ for sufficiently small $c>0$ (say, $\sim 1/1000$), such that $1/\eps$ is an integer, and $q := \log{\frac{1}{\xi}}$.
Clearly, conditions $\eps \in (0, 1)$ and $q > 2$ required by Croot-Laba-Sisask Lemma are satisfied.\\
This gives a function 
$g(x) := \frac{1}{d}(c_1\gamma_1(x) + \ldots + c_d\gamma_d(x))$ with $|c_j| \leqslant 1$ and $\gamma_1, \ldots, \gamma_d$ being some elements in $\hat{G}$, such that $\|f/\omega - g\|_{L^q} \leqslant \eps$, and $d \leqslant Cq/\eps^2  \ll \frac{\omega^2}{\delta^2}\log{\frac{1}{\xi}}$. We now adapt $g$ to satisfy the conditions in Theorem's statement. \\
Consider any $\eps$-$d$-block $B$, with respect to homomorphisms $\gamma_1, \ldots, \gamma_d$. 
It is rather trivial to check that $g$ varies by at most $O(\eps)$ on $B$. Indeed, if $x, y$ belong to the same block, we have
$$
|g(x) - g(y)| = 
\frac{1}{d} \big|\sum_{i=1}^d c_i (\gamma_i(x) - \gamma_i(y)) \big| \leqslant
\frac{1}{d}\sum_{i=1}^d|c_i|\frac{2\pi}{K} \leqslant
\frac{1}{d}\frac{2\pi d}{K} = 2\pi \eps,
$$
where $K := 1/\eps$.
Let us now update values of $g$ by at most $2\pi \eps$, such that $g$ is now a constant on each block, and this constant belongs to $\eps \Z$. Clearly, now $\|f/\omega - g\|_{L_q} \leqslant (2\pi + 1)\eps < 10\eps$.\\
Define $h$ as $h := \omega(f/\omega - g)$. 
Let $Y$ be a set of those elements $y$ which do not satisfy $h(y) \leqslant 100\omega\eps$. Since $\|h\|_{L_q} \leqslant 10\omega\eps$ we see that
$
|Y|(100\omega\eps)^q \leqslant \sum_{x}h(y)^q \leqslant |G|(10\omega\eps)^q,
$
which gives $|Y| \leqslant |G|/10^q < |G|/e^q = \xi |G|$.\\
Let $k$ be a functions equal $h$ on $Y$ and identically zero elsewhere. We update $h$ such that equality $f = \omega g + h + k$ still holds; clearly, $\|h\|_{\infty} \leqslant 100\omega\eps \leqslant \delta/10$ and $|\supp{k}| = |Y| \leqslant \xi |G| $.

\end{proof}

{\it Remark.} Note that this result is essentially equivalent to the Croot-Laba-Sisask Lemma.\\
{\it Remark.} If instead of applying Croot-Laba-Sisask Lemma we have directly considered the decomposition $f = \sum_{\gamma}\hat{f}(\gamma)\gamma$, we then would still have a meaningful result, but with bound $d \ll \frac{\omega^2}{\delta^2}\frac{1}{\xi}$ instead of $d \ll \frac{\omega^2}{\delta^2}\log{\frac{1}{\xi}}$.

\subsection{Wrapping level-sets of functions}

\begin{corollary}\label{cor:wrapping}
Let $f : G \rightarrow \R$ be a function of Wiener norm $\omega$. Let $\eta_1, \eta_2$ be arbitrary reals, and let $0 < \xi < 1/100$ and $\delta > 0$ be parameters. 
Then there exists $\eps$-$d$-wrapper $W$ and set $Y$ such that $|Y| \leqslant \xi |G|$, and all $x \not \in Y$ satisfy the following:
\begin{enumerate}
    \item If $l_1 \leqslant f(x) \leqslant l_2$, then $x \in W$,
    \item If $f(x) \leqslant l_1 - \delta$ or $f(x) \geqslant l_2 - \delta$, then $x \not\in W$.
\end{enumerate}
Moreover,
$
\eps \gg \min(\frac{\delta}{\omega}, 1),\ \  
d \ll \max(\frac{\omega^2}{\delta^2}, 1)\log{\frac{1}{\xi}}.
$
\end{corollary}
\begin{proof}
If $\delta > \omega$, we replace it with $\delta := \omega$, since it does not affect the statement.\\
Application of the Wrapper Theorem with the parameters $f := f$, $\delta := \delta$ gives a decomposition $f = g + h + k$. Setting $Y := \supp k$, one obtains $|Y| \leqslant \xi|G|$ . 
Now consider a set 
$$
W := \{ x : l_1 - 0.5\delta \leqslant g(x) \leqslant l_2 + 0.5\delta\}.
$$
Since $g$ is constant on each $\eps$-$d$-block, if one element of the block is in $W$, then the whole block is there. 
Thus it is clear that $W$ is a union of $\eps$-$d$-blocks, and therefore is an $\eps$-$d$-wrapper.\\
Let us demonstrate $W$ satisfies its declared properties. \\
Let us take some $x$ outside of the set $Y$. Clearly, $f(x) = g(x) + h(x)$ in this case, and therefore $|f(x) - g(x)| \leqslant 0.1\delta$. \\
If $x$ is such that $l_1 \leqslant f(x) \leqslant l_2$, then $l_1 - 0.1\delta \leqslant g(x) \leqslant l_2 + 0.1\delta$, and thus $x \in W$.\\
If $x$ is such that $f(x) \geqslant l_2 + \delta$, then $g(x) \geqslant l_2 + 0.9\delta$, and thus $x \not\in W$. Same for $f(x) \leqslant l_1 - \delta$.
\end{proof}

\subsection{Properties}

Recall some facts on Wiener norm of sets and functions.
\begin{proposition}
\label{prop:win-pr}
Let $P$ be an arithmetic progression in $\F_p$. Then $\|P\|_\omega \ll \log{p}$.
\end{proposition}

\begin{proposition}
\label{prop:win-int}
Let $X, Y$ be 2 subsets of $G$. Then
$\|X \cap Y\|_\omega \leqslant \|X\|_\omega \|Y\|_\omega$.
\end{proposition}
\begin{proof}
It is easy to see that for any $\xi \in \hat{G}$ equality
$
\FF{XY}{(\xi)} = \frac{1}{|G|}\sum_{\eta \in \hat{G}}\hat{X}(\xi - \eta)\hat{Y}(\eta)
$
takes place, from where it is easy to imply the Proposition's statement.
\end{proof}

\begin{definition}
For $X$ an arbitrary arc on $\cS$, and $\psi : G \rightarrow \cS$ an arbitrary homomorphism (possibly trivial) we set $\psi^{-1}(X)$ to be those elements $g \in G$, such that $\psi(g) \in X$. Denote
$$
w(G) := 
\max_{\substack{\psi - \text{homomorphism}, \\ X - \text{arbitrary arc}}}
\|\psi^{-1}(X)\|_{\omega}.
$$
\end{definition}
Note that $w({\Z_N}) \ll \log{N}, w({\Z_p^N}) \ll p$. In general, $w(G) \ll \exp(\log^{1/2 + o(1)}{|G|})$.\\
The preimage of $\cS_i$ under any homomorphism has Wiener norm at most $w(G)$, and any $\eps$-$d$-block is an intersection of $d$ preimages. This implies Property \ref{property:wiener_bound}:
\begin{proposition}
Let $B$ be an $\eps$-$d$-block in abelian group $G$. Then 
$
\|B\|_{\omega} \leqslant \omega(G)^d.
$
\end{proposition}
To prove it, we just exploit Proposition \ref{prop:win-int}. 
\begin{proposition}
\label{prop:intersec}
Let $W$ be an $\eps$-$d$-wrapper in abelian group $G$. Then
$
\|W\|_\omega \leqslant (\omega(G) / \eps )^d.
$
\end{proposition}
To prove it, we just exploit the triangle inequality:
$$
\|W\|_{\omega} \leqslant
\sum_{\substack{B \text{ is an } \eps-d-\text{block},\\ B \subseteq W}}
\|B\|_{\omega}
\leqslant 
\#\{B \subseteq W, B \text{ is an  $\eps$-$d$-block}\}\omega(G)^d 
\leqslant
(1/\eps)^d\omega(G)^d.
$$
Thus, the Wiener norm of a wrapper is also bounded in terms of $\eps, d$, and $\omega(G)$.\\
Recall a Weil's bound for Kloosterman sum \cite{weil}. For any $a, b \neq 0$,
$
\bigg| \sum_{z \neq 0}e_p(az + bz^*) \bigg| \leqslant 2\sqrt{p}.
$
Using Weil bound, we derive the following:
\begin{proposition}
Let $X \subseteq \fpstar$. Then inequality
$\max_{\xi \neq 0}|\FF{X^{*}}(\xi)| \leqslant 2\sqrt{p}\|X\|_\omega$ holds.
\end{proposition}
\begin{proof}
$$
\FF{X^*}(\xi) =
\sum_{z}X^*(z)e_p(-\xi z) =
\sum_{z}X(z)e_p(-\xi z^*) = 
$$
$$
=
\frac{1}{p}
\sum_{z} \bigg( \sum_{\eta}\hat{X}(\eta)e_p(\eta z)\bigg)e_p(-\xi z^*) = 
\frac{1}{p}
\sum_{\eta}\hat{X}(h)\sum_{z}e_{p}(\eta z - \xi z^*),
$$
which obviously does not exceed $2\sqrt{p}\|X\|_\omega$.
\end{proof}
Now we prove Property \ref{property:wcapw}:
\begin{proof}
By Fourier Transform properties, 
$$
|W_1 \cap W_2^*| = 
\sum_{x}W_1(x)W_2^*(x) = 
\frac{1}{p}
\sum_{h}\hat{W_1}(h)\hat{W_2^*}(h) =
\frac{|W_1| |W_2|}{p} + 
\frac{1}{p}
\sum_{h\neq 0}\hat{W_1}(h)\hat{W_2^*}(h).
$$
Therefore,
$$
\bigg|\big|W_1\cap W_2^*| - \frac{|W_1||W_2|}{p}\bigg|
\leqslant
\frac{1}{p}\max_{h\neq 0}|\hat{W_2^*}(h)|
\sum_{h\neq 0}|\hat{W_1}(h)|
\leqslant 2\sqrt{p}\omega_2\omega_1.  
$$
\end{proof}

\section{Proof of Main and Supplementary Lemmas}
\label{sec:main_supp_proofs}

\subsection{Proof of the Supplementary Lemma}
\begin{proof}
Let $A, B \subseteq G, |G| = N, |A| = \alpha' N, |B| = \beta' N, \alpha' \geqslant \alpha, \beta' \geqslant \beta$. Since 
$$ 
\|A * B \|_{\omega} = 
\frac{1}{p}
\sum_{\gamma \in \hat{\F}_p}|\hat{A}(\gamma)| |\hat{B}(\gamma)|
\leqslant 
\frac{1}{p}
\sqrt{\sum_{\gamma \in \hat{G}}|\hat{A}(\gamma)|^2}
\sqrt{\sum_{\gamma \in \hat{G}}|\hat{B}(\gamma)|^2}
=
\sqrt{|A| |B|},
$$
we obtain the inequality $\|A * B\|_{\omega} \leqslant \sqrt{\alpha'\beta'}N$. Applying Corollary \ref{cor:wrapping} with $f = A * B$, $l_1 := \eta_1 |G|, l_2 := \eta_2 |G|$, and $\delta := \delta N$, we obtain the Supplementary Lemma.
\end{proof}

\subsection{Proof of the Main Lemma}
We can deduce the proof in a way similar to the one in the proof of Technical Lemma. However, this would involve tedious calculations, and we proceed with a cleaner one.

\begin{proof}
Without loss of generality, we assume that $(A_1 * \ldots * A_k)(0) = o(p^{k-1})$.

Let us fix some positive integer $d$ and rearrange frequences $\gamma_i$ such that $|\hat{A_1}(0)| > |\hat{A_1}(\gamma_1)| \geqslant |\hat{A_1}(\gamma_2)| \geqslant \ldots$. 
Clearly, since $\sum_{i}|\hat{A}_1(\gamma_i)|^2 = p|A| < p^2$, one has $|\hat{A}(\gamma_d)| < p/\sqrt{d}$ for any $d$.
Now write
$$
\#\{(a_1, \ldots, a_k) \in A_1 \times \ldots \times A_k: a_1 + \ldots + a_k = 0\} = 
\frac{1}{p}
\sum_{i = 1}^{d}\hat{A_1}(\gamma_i)\ldots \hat{A_k}(\gamma_i) 
+
\frac{1}{p}
\sum_{i > d}\hat{A_1}(\gamma_i)\ldots \hat{A_k}(\gamma_i).
$$
We estimate the second term as 
$
\frac{1}{p}\max_{i > d}|\hat{A_1}(\gamma_i)| \sum_{i}|\hat{A_2}(\gamma_i)|
\ldots|\hat{A_k}(\gamma_i)| 
\leqslant 
\frac{1}{p}\frac{p}{\sqrt{d}}p^{k-1} = p^{k-1}/\sqrt{d}.
$
Let us now introduce the parameter $\eps > 0$ such that $K := 1/\eps$ is a positive integer, and let split the unit circle $\{z : |z| = 1\}$ into $K$ equal arcs $\cS_1, \cS_2, \ldots, \cS_K$ such that $\cS_1$ is symmetric around $z=1$. Let us introduce the Bohr set: 
$$
\cB := \{x : \gamma_i(x) \in \cS_1, 1 \leqslant i \leqslant d\}.
$$
Clearly, for any $\gamma_i, 1 \leqslant i \leqslant d$, $ \frac{\hat{\cB}(\gamma_i)}{|\cB|} = 1 + O(\eps)$ holds.\\
Let $|\cB| = \theta p$. 
It is well-known, that Bohr set satisfies the inequality $\theta \geqslant \eps^d p$.\\
Now we introduce functions $f_1, \ldots, f_k: \F_p \rightarrow \C$ defined by
$$
f_i(x) := \frac{(A_i * \cB)(x)}{|\cB|} \text{ for } 1 \leqslant i \leqslant k.
$$

These functions are `smoothed' versions of sets $A_1, \ldots, A_k$:
$$
(f_1 * \ldots * f_k)(0) = 
\frac{1}{p}\sum_{i}
\hat{A_1}(\gamma_i)\ldots \hat{A_k}(\gamma_i) \frac{\hat{\cB}(\gamma_i)^k}{|\cB|^k} = 
\frac{1}{p}\sum_{i = 1}^{d}
\hat{A_1}(\gamma_i)\ldots \hat{A_k}(\gamma_i)
+ O\big(p^{k-1}/\sqrt{d}\big) 
+ O\big(\eps p^{k-1}\big).
$$
To make both error terms equally small, we set $\eps \sim 1/\sqrt{d}$, so $(f_1 * f_2 * \ldots * f_k)(0) \ll \eps p^{k-1}$.

Let us work with level-sets of functions $f_i$. Take some particular $f_i$.
Clearly, Wiener norm of the function $|\cB|f_i = A_i * \cB$ is at most $\sqrt{\alpha_i \theta}p$.
Let us choose some sufficiently small parameters $\chi_i, \xi_i$. 
We apply Corollary 1 with $\eta := \delta := \chi_i \alpha_i \theta p$, which gives an $\eps_i$-$d_i$-wrapper $W'_i$ and set $Y'_i$, such that all $x$ outside $Y'_i$ satisfy the following:
\begin{enumerate}
    \item if $f_i(x) > 2\chi_i\alpha_i$, then $x \in W'_i$,
    
    \item if $f_i(x) < \chi_i \alpha_i$, then $x \not \in W'_i$.
\end{enumerate}

Moreover,
$
\eps_i \gg \chi_i \sqrt{\alpha_i \theta} \ \ 
d_i \ll \frac{1}{\chi_i^2 \alpha_i \theta}
\log{\frac{1}{\xi_i}}, \ \ 
|Y'_i| \leqslant \xi_i p.
$ \\
Point (2) implies an inequality 
$
W'_i(x) \leqslant \frac{1}{\chi_i \alpha_i}f_i(x) + Y'_i(x)\ \ \forall x \in \F_p.
$
Therefore, 
$$
(W'_1 * \ldots * W'_k)(0) = 
\sum_{x_1 + \ldots + x_k = 0} W'_1(x_1) \ldots W'_k(x_k)
\leqslant
\sum_{x_1 + \ldots + x_k = 0} \prod_{i}\bigg( \frac{1}{\chi_i \alpha_i}f_i(x) + Y'_i(x)\bigg)
\leqslant
$$
$$
\leqslant 
\frac{(f_1 * \ldots * f_k)(0)}{\chi_1 \ldots \chi_k \alpha_1 \ldots \alpha_k} 
 + 
O(\ldots) 
\ll
\frac{\eps}{\chi_1 \ldots \chi_k \alpha_1 \ldots \alpha_k} p^{k-1}.
$$
We suppressed the second error term, since it is negligebly small due to choice of parameters $\xi_i$.\\
To make this quantity equal $o(p^{k-1})$, we need $d$ to dominate the value of $1/\chi_1 \ldots \chi_k \alpha_1 \ldots \alpha_k$. To make Wiener norms of sets $W'_i$ to be at most $p^{0.001}$, we need inequalities $d_i\log{1/\eps_i} \ll \log{p}$, or $(\chi_i)^{-2}\alpha_i^{-1}\theta^{-1}\log{\xi_i^{-1}} \ll \log{p}/\log\log{p}$. Since $\theta \gg \eps^d$, this would follow from $d \ll \log\log{p}/\log\log\log{p}$ and $\xi_i \gg 1/(\log{p})^{100}$.\\
Now let us construct particular $W_i, Y_i$, required by the Lemma's statement.\\
There exists $b_i \in \cB$, such that $f(a + b_i) \geqslant 2\chi_i \alpha_i$ for all but $2\chi_i |A_i|$ elements $a$ of $A_i$. Indeed,
$$
|\cB|\min_{b \in \cB} \#\{a \in A : f(a + b) < 2\chi_i \alpha_i\} \leqslant
\sum_{b \in \cB} \#\{a \in A : f(a + b) < 2\chi_i \alpha_i\} =
$$
$$
= 
\sum_{x: f(x) < 2\chi_i \alpha_i} (A * \cB)(x) 
<
p \cdot 2\chi_i \alpha_i |\cB| = 2\chi_i |A||\cB|
$$
Let us denote this particular $b$ as $b_i$, and denote by $Y''_i$ those $a \in A$ which do not satisfy  $f(a + b_i) \geqslant 2\chi_i \alpha_i$. Clearly, if $a \in A$ is not in $Y'_i$ and not in $Y''_i - b_i$, then $f(a + b_i) > 2\chi_i \alpha_i$, and point (1) implies $a + b_i \in W'_i$, or $a \in W'_i - b_i$.
Setting $Y_i := E'_i \cup (E''_i - b_i)$ and $W_i := W'_i - b_i$ gives the sets, required by the Main Lemma statement. \\
Point (1) of the Lemma statement is satisfied by the $(W_1 * \ldots * W_k)(b_1 + \ldots + b_k) = o(p^{k-1})$.\\
Point (2) is satisfied by $|Y_i| \ll \chi_i |A| + \xi_i p = o(p)$. \\
Point (3) is satisfied by the fact that $W_i$ has bounded dimension and granularity, which makes the argument of the Property \ref{property:wcapw} applicable. This completes the proof.
\end{proof}
{\it Remark.}
Inequalities, showing up in the course of the proof, show that one can make the statement lemma quantitative with densities of $\alpha_i$ being some small degrees of $1/\log\log{p}$. \\
If the Croot-Laba-Sisask result is applied towards the sumset $A_i * \cB$, one can improve a small degree of $1/\log\log{p}$ to some small degree of $1/\log{p}$. However, `asymmetric' results such as the Technical Lemma above appear to give better estimates in our particular applications, and therefore we do not prove a better quantitative symmetric result here.

\noindent{
A. S. Semchankau, \\
Steklov Mathematical Institute of Russian Academy of Sciences,\\
8 Gubkina St., Moscow 119991, Russia}\\
{\tt aliaksei.semchankau@gmail.com}

\end{document}